\documentclass{article}
\usepackage{ulem}

\usepackage{amssymb,amsmath,graphicx,amsthm,mathrsfs,bm}
\usepackage[textsize=footnotesize,color=DarkGreen!40]{todonotes}

\usepackage[colorlinks=true]{hyperref}
\usepackage{pdfsync,color}
\numberwithin{equation}{section}
\newtheorem{theorem}{Theorem}[section]

\newtheorem{Lemma}[theorem]{Lemma}

\newtheorem{proposition}[theorem]{Proposition}
\theoremstyle{definition}
\newtheorem{Definition}[theorem]{Definition}
\newtheorem{corollary}[theorem]{Corollary}
\theoremstyle{definition}
\newtheorem{remark}[theorem]{Remark}
\numberwithin{equation}{section}

\newcommand{\lc}
{\mathrel{\raise2pt\hbox{${\mathop<\limits_{\raise1pt\hbox
{\mbox{$\sim$}}}}$}}}

\newcommand{\gc}
{\mathrel{\raise2pt\hbox{${\mathop>\limits_{\raise1pt\hbox{\mbox{$\sim$}}}}$}}}

\newcommand{\ec}
{\mathrel{\raise2pt\hbox{${\mathop=\limits_{\raise1pt\hbox{\mbox{$\sim$}}}}$}}}

\def\bb{\begin{equation}} \def\ee{\end{equation}}

\def\beqn{\begin{eqnarray}}  \def\eqn{\end{eqnarray}}

\def\beqnx{\begin{eqnarray*}} \def\eqnx{\end{eqnarray*}}

\def\bn{\begin{enumerate}} \def\en{\end{enumerate}}

\def\bd{\begin{description}} \def\ed{\end{description}}

\def\label{\label}

\renewcommand{\leq}{\leqslant}
\renewcommand{\geq}{\geqslant}

\title{A uniform bound on costs of controlling semilinear heat equations on a sequence of increasing domains and its application}

\author{
Lijuan Wang\thanks{School of
Mathematics and Statistics, Wuhan University; Computational Science Hubei Key Laboratory,
Wuhan University, Wuhan 430072, China;
e-mail: ljwang.math@whu.edu.cn.},
\quad Can Zhang
\thanks{Corresponding author. School of Mathematics and Statistics, Wuhan University; Computational Science Hubei Key Laboratory,
Wuhan University, Wuhan 430072, China;
e-mail: canzhang@whu.edu.cn.}}

\begin{document}

\date{}

\maketitle

\begin{abstract}
In this paper, we first prove a uniform upper bound on costs of null controls for semilinear heat equations with globally Lipschitz nonlinearity on a sequence of increasing domains, where the controls are acted on an equidistributed set that spreads out in the whole Euclidean space $\mathbb R^N$.
As an application, we then show the exactly null controllability for this semilinear heat equation in  $\mathbb R^N$.
The main novelty here is that the upper bound on costs of null controls for such kind of equations in large but bounded domains can be made uniformly with respect to the sizes of domains under consideration. The latter is crucial when one uses a suitable approximation argument to derive the global null controllability for the semilinear heat equation in $\mathbb R^N$.
This allows us to overcome the well-known problem of the lack of compactness embedding arising in the study of null controllability for nonlinear PDEs in generally unbounded domains.
\end{abstract}

\medskip

\noindent\textbf{2010 Mathematics Subject Classifications.}
35K05, 93B07, 93C20

\medskip

\noindent\textbf{Keywords.}
Semilinear heat equation, null controllability, uniform cost, equidistributed set

\section{Introduction and main results}

This paper is concerned with the null control costs for semilinear heat equations on a sequence of increasing bounded  domains in $\mathbb R^N$ (with $N\in\mathbb N$), when the controls act on the interior subsets of these domains.  Generally speaking, the null control costs depend on the geometry of both control regions and whole domains where the equations evolve.
Nevertheless, the goal of this paper is to investigate the uniform upper estimate for null control costs with respect to
the varying domains.  As an interesting  application, we shall derive
the null controllability of the semilinear heat equation in $\mathbb R^N$.

The general formulation of the problem could be stated as follows. Let $T$ be a positive time and let $E$ be a subset of positive Lebesgue measure in $(0,T)$.
For each $n\in\mathbb N$, let
$\omega_n$ be a nonempty open subset of bounded Lipschitz domain $\Omega_n$ in $\mathbb R^N$.
Consider
the following semilinear heat equation:
\begin{equation}\label{10301}
\left\{
\begin{array}{lll}
\partial_t z_n-\Delta z_n+f(z_n)=\chi_{\omega_n}\chi_E  u_n&\mbox{in}&\Omega_n\times (0,T),\\
z_n=0&\mbox{on}&\partial\Omega_n\times (0,T),\\
z_n(0)=z_0|_{\Omega_n},
\end{array}\right.
\end{equation}
where $z_0\in L^2(\mathbb{R}^N)$,  $u_n=u_n(x,t)\in L^2(\Omega_n\times(0,T))$ is the control function,
$\chi_E$ and $\chi_{\omega_n}$
are characteristic functions of $E$ and control region $\omega_n$, respectively. Throughout the paper, we  suppose that
$f$ is a globally Lipschitz function.  Then, there exists a unique solution
$z_n\in C([0,T]; L^2(\Omega_n))\bigcap L^2(0,T;H_0^1(\Omega_n))$ for the
equation $\eqref{10301}$ (c.f., \cite{Fabre}, for instance).

Recall that for each $n\in\mathbb N$, the control system $\eqref{10301}$ is called  exactly null controllable at the time $T>0$, if the following statement is true:
For any $z_0\in L^2(\mathbb R^N)$, there is a control
$u_n\in L^2(0,T;L^2(\Omega_n))$ so that the corresponding solution $z_n$ to $\eqref{10301}$ satisfies that  $z_n(T)=0$ in $\Omega_n$. Furthermore, the null control cost is the least constant $C_n$ so that $\|u_n\|_{L^2(\Omega_n\times(0,T))}\leq C_n\|z_0\|_{L^2(\mathbb R^N)}$
holds for all $z_0\in L^2(\mathbb R^N)$.  As mentioned at the beginning, the constant $C_n$ may
depend on the geometry parameters  of $\Omega_n$ and $\omega_n$.

The question whether or not  a given control system is null controllable and obtaining upper bounds for the associated null control costs,
with respect to the time interval or control regions,
 are important topics in control theory, both for linear partial differential control equations and abstract linear control systems. We refer the reader to \cite{AEWZ,B,EV17,Mat,ph2,WangZhangZhang,wc,zxz} and references therein for a wider discussion on this activated research field.

Meanwhile, there are many  fascinating  works in the literature on the approximate or exact  null controllability for the semilinear heat equations in bounded domains; see, for instance, \cite{Fabre,FI,Phung-Wang-Zhang}.  Their proofs are usually divided into two parts:
(i) null controllability of the linearized system;
(ii) a fixed-point argument.
The globally null controllability was also proved for a class of nonlinearities for which blow-up phenomena may arise (see, e.g., \cite{B1,fz,LB20}).

To the best of our knowledge, there are very few works on establishing a uniform upper bound for the null control cost for a partial differential controlled equation in varying domains.

In this paper, we shall restrict the control on an equidistributed set.  
We say a set as an equidistributed set in $\mathbb R^N$ if it contains a union of suitably distributed balls of fixed radius. Recently, there are many beautiful existing results on the quantitative unique continuation for
general elliptic operators on  equidistributed sets (c.f., \cite{Mat,SV20} and references therein).
Meanwhile,  we will consider the varying domains $\Omega_n$ as an approximation of $\mathbb R^N$, for the simplicity.

In order to state our main results, we first introduce certain standard notations.
For each $r>0$ and $x_{0}\in \mathbb{R}^{N}$,
$B_{r}(x_{0})$ stands for the closed ball centered  at $x_{0}$  and of radius $r$;
$Q_{r}(x_{0})$ denotes the smallest closed cube centered at $x_0$
so that $B_{r}(x_{0})\subset Q_{r}(x_{0})$;
$\mathrm{int}(Q_{r}(x_{0}))$ is  the interior of $Q_{r}(x_{0})$.

Let $0<r_1<r_2<\infty$. The following three assumptions will be effective throughout the paper:

$(H_1)$.  There is a sequence
$\{x_i\}_{i=1}^\infty\subset\mathbb R^N$ so that $\mathbb{R}^{N}= \displaystyle{\bigcup_{i=1}^\infty}Q_{r_{2}}(x_{i})$,
$\mathrm{int}(Q_{r_{2}}(x_{i}))\bigcap \mathrm{int}(Q_{r_{2}}(x_{j}))=\emptyset$
for each $i\neq j$.
Moreover, $\omega\triangleq\displaystyle{\bigcup_{i=1}^\infty}\omega_{i}$, where
$\omega_{i}$ is an open set and $B_{r_{1}}(x_{i})\subset \omega_{i} \subset B_{r_{2}}(x_{i})$ for each $i\in\mathbb N$.

\medskip

$(H_2)$.  For each $n\in\mathbb N$, $\Omega_n\triangleq \mbox{int}\Big(\displaystyle{\bigcup_{i\in I_n}} Q_{r_2}(x_i)\Big)$ is convex,
with $\mbox{Card}\,(I_n)<\infty$. In addition,  $I_n\subsetneq I_m$ when $n<m$, and $\displaystyle{\bigcup_{n=1}^\infty}\Omega_n=\mathbb{R}^N$.

\medskip

$(H_3)$.  The function  $f: \mathbb{R}\rightarrow \mathbb{R}$  is globally Lipschitz continuous, i.e., $|f(s)-f(\tau)|\leq L|s-\tau|$ for all $s,\tau\in\mathbb R$ with some constant $L>0$, and satisfies
that $f(0)=0$.\\

With regarding to the assumption $(H_1)$, we may say that the set $\omega$ is an equidistributed set in $\mathbb R^N$. A particularly example  is a periodic arrangement of balls  (see Figure \ref{fig1} below).
\begin{figure}[h]
\centering
\includegraphics[width=3cm]{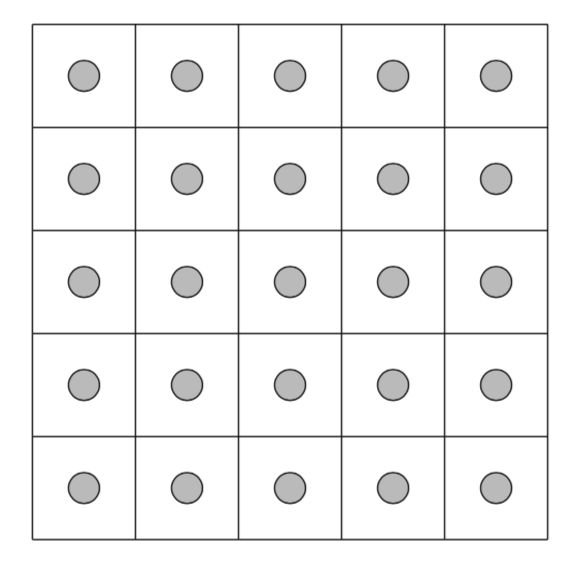}
\caption{Example}
\label{fig1}
\end{figure}
Very recently, we have proved in \cite{duan-wang-zhang} the observability inequality and null controllability  on such  kind of sets  for the linear heat equation with time and space dependent potentials in  $\mathbb R^N$.

The first main result of this paper concerning the uniform upper bound on costs of controlling a semilinear heat equation on increasing large domains can be stated as follows.

\begin{theorem}\label{null-3}
Let $(H_1), (H_2)$ and  $(H_3)$ hold. Let $T>0$ and $E$ be a subset of positive measure in $(0,T)$.
Then there is a positive constant $\kappa= \kappa(r_1,r_2,E,T,L)$ so that for any $n\in\mathbb N$ and any $z_0\in L^2(\mathbb R^N)$, there is a control $u_n\in L^2(\mathbb R^N\times (0,T))$, with the uniform bound
$$
\|u_n\|_{L^2(\mathbb R^N\times (0,T))}\leq \kappa\|z_0\|_{L^2(\mathbb R^N)},
$$
so that the corresponding solution $z_n\in C([0,T];L^2(\Omega_n))$ to the following semilinear  heat equation
\begin{equation*}
\left\{
\begin{array}{lll}
\partial_t z_n-\Delta z_n+f(z_n)=\chi_{\omega\cap\Omega_n} \chi_E u_n&\mbox{in}&\Omega_n\times (0,T),\\
z_n=0&\mbox{on}&\partial\Omega_n\times (0,T),\\
z_n(0)=z_0&\mbox{in}&\Omega_n,
\end{array}\right.
\end{equation*}
satisfies that $z_n(T)=0$ over $\Omega_n$.
\end{theorem}

\begin{remark}
Our argument does not allow us to establish a similar result for  the global null controllability of slightly superlinearities as studied in \cite{B1,fz,LB20}.
\end{remark}

\begin{remark}
We do not know how to extend this result from an equidistributed set $\omega$ to a more general thick set.  We refer the reader to \cite[Remark 1.7]{duan-wang-zhang} for the difficulty.
\end{remark}

Another motivation of this paper is to establish the null controllability for the semilinear
heat equation in the whole space $\mathbb R^N$,  when the control is acted on an equidistributed set.
As already remarked in Remark 3.5 of our recent work \cite{duan-wang-zhang} on the null controllability for the linear heat equation with bounded potentials, the linearized plus fixed-point approach in \cite{Fabre,FI} cannot be directly applied in the framework of general unbounded domains  because of the lack of compactness of Sobolev's embedding.

The authors of \cite{TeresaDe} studied
the approximate controllability of a semilinear heat equation in an unbounded domain $\mathcal O$ of
$\mathbb R^N$, with control only acted in an open and nonempty subset, by an approximation method.  More precisely, they first considered the approximate controllability problem in bounded domains of the form $\mathcal O_n\triangleq\mathcal O\bigcap B_n$, where $B_n$ denotes the ball centered at the origin and of radius $n$; and then they showed that the controls proposed in
\cite{Fabre} restricted to $\mathcal O_n$ converge  in some sense to a desired approximate control in
$\mathcal O$, as $n$ goes to infinity. One of main ingredients in their proofs is a qualitative    unique continuation property for the linear  parabolic equation.  Nevertheless, the technical proof of \cite{TeresaDe} (see also \cite{ldt}) is not valid any more for the null controllability of a semilinear heat equation in $\mathbb R^N$.

It is also mentioning that the authors in \cite{CMZ} and \cite{Gon-De} proved the null controllability for some semilinear heat equations in an unbounded domain $\mathcal O$ of $\mathbb R^N$,  when the control is assumed to be distributed along a subdomain $\omega$ so that the uncontrolled region $\mathcal O\setminus\omega$
is bounded.  The arguments therein are mainly based on a Carleman estimate for the linear parabolic operator in bounded domains.

Recently, the authors in \cite{SV20} showed in a linear and  abstract framework  that if the sequence of null controls associated to an exhaustion of an unbounded domain converges, then the solutions converge in the same way to the limiting problem on the unbounded domain.
This result allows to infer the null controllability on unbounded domain by studying the control problem on a sequence of bounded domains. In  particular, it recovers the null controllability result for the linear heat equation in $\mathbb R^N$. The latter has been
already established well in \cite{EV17} and \cite{WangZhangZhang} independently.

Inspired by these works, we could utilize Theorem \ref{null-3}  to prove the null controllability for a semilinear heat equation in
$\mathbb{R}^N$ with the control acted on an equidistributed set $\omega$.
In fact,  since null controls
in Theorem \ref{null-3}
are uniformly bounded,  $\{u_n\}_{n\geq1}$ has a weakly convergent subsequence with a limit $u$ in $L^2(\mathbb R^N\times(0,T))$.  Thus, one may expect that such control maybe a null control for the corresponding semilinear
heat controlled equation in $\mathbb R^N$. Actually we have the following result:

\begin{theorem}\label{null-2}
Assume that $(H_1)$ and $(H_3)$ hold.  Let $T>0$ and $E$ be a subset of positive measure in $(0,T)$.
Then, for each initial value $y_0\in L^2(\mathbb R^N)$, there is a control
$u\in L^2(0,T;L^2(\mathbb{R}^N))$
with an upper  bound
$$\|u\|_{L^2(\mathbb R^N\times (0,T))}\leq \kappa\|y_0\|_{L^2(\mathbb R^N)},$$
so that the corresponding solution
\begin{equation}\label{null-wang}
\left\{
\begin{array}{lll}
\partial_t y-\Delta y+f(y)=\chi_\omega \chi_E u&\mbox{in}&\mathbb{R}^N\times (0,T),\\
y(0)=y_0&\mbox{in}&\mathbb{R}^N,
\end{array}\right.
\end{equation}
satisfies that $y(T)=0$ in $\mathbb{R}^N$.   Here $\kappa$ is the same constant as in Theorem \ref{null-3}.

\end{theorem}

\begin{remark}
The well-posedness of such kind of semilinear heat equations in general unbounded domains is stated without a proof in \cite{CMZ}, for instance.  We refer the reader to Definition \ref{weak solution} and Lemma \ref{Existence} in Section \ref{b5} for precise presentations.
\end{remark}

\begin{remark}
Instead of $\mathbb R^N$, our considerations are valid as well for every set that can be
approximated with cubes such as the half space $\mathbb R^{N-1}\times\mathbb R^+$
and the infinite strip.
\end{remark}

\begin{remark}
Our method does not rely on any compactness argument in the whole space $\mathbb R^N$.
\end{remark}

The rest of this paper is organized as follows. Section \ref{b1} is devoted to prove the uniform bound of  control costs for the linearized controlled system.
Section \ref{b2} shows the proof of Theorem \ref{null-3} by a fixed-point argument.  Finally, Section \ref{b5} presents the proof of Theorem \ref{null-2}.

\medskip

\section{The linear case}\label{b1}

For each $n\in\mathbb N$, let $\varphi_n$ be the solution to
the following linear heat equation with a bounded potential $a \in L^\infty(\mathbb R^N\times(0,T))$:
\begin{equation}\label{1.1}
\left\{ \begin{array}{lll}
\partial_{t}\varphi_n-\Delta\varphi_n+a\varphi_n=0& \mathrm{in}& \Omega_n\times (0,T),\\
\varphi_n=0& \mathrm{on}& \partial\Omega_n\times (0,T),\\
\varphi_n(0)=\varphi_{0}& \mathrm{in}& \Omega_n,\\
\end{array}\right.\end{equation}
with $\varphi_0\in L^2(\mathbb R^N)$.

We first state a uniform observability inequality for all solutions of \eqref{1.1} evolving in $\Omega_n$ for all $n\in\mathbb N$.  Here and throughout this paper, we simply write $\|a\|_{\infty}=\|a\|_{L^{\infty}(\mathbb R^N\times(0,T))};$ and write
$C(\cdot)$  for a  positive constant depending on what are enclosed in the brackets.

\begin{theorem}\label{Thm1}
Assume that $(H_1)$ and $(H_2)$ hold. Let $T>0$ and $E$ be a subset of positive measure in $(0,T)$.
Then there are constants $C=C(r_{1}, r_{2})$ and
$\widetilde{C}=\widetilde{C}(r_{1}, r_{2},E)$
so that  the following  observability inequality
$$
\int_{\Omega_n} |\varphi_n(x,T)|^2\,\mathrm dx\leq
e^{\widetilde{C}}e^{C\left(T+T\|a\|_{\infty}+\|a\|_{\infty}^{2/3}\right)}
\int_E\int_{\omega\bigcap \Omega_n}|\varphi_n(x,t)|^{2}\mathrm dx\mathrm dt,\;\;\;\;\forall \varphi_0\in L^2(\mathbb R^N),
$$
holds uniformly for all $n\in\mathbb N$.
\end{theorem}

\begin{remark}
Note that the two constants in the above theorem are independent of the sizes of  domains.
\end{remark}

\begin{remark}\label{Thm2}
Under the same assumptions as in Theorem \ref{Thm1}, the following refined estimate is also true:
There are constants $C=C(r_{1}, r_{2})$ and
$\widetilde{C}=\widetilde{C}(r_{1}, r_{2},E)$
so that  the $L^1$-type observability inequality
$$
\|\varphi_n(T)\|_{L^2(\Omega_n)}\leq
e^{\widetilde{C}}e^{C\left(T+T\|a\|_{\infty}+\|a\|_{\infty}^{2/3}\right)}
\int_E \int_{\omega\bigcap \Omega_n} |\varphi_n(x,t)|\mathrm dx\mathrm dt,\;\;\;\;\forall \varphi_0\in L^2(\mathbb R^N),
$$
holds uniformly for all $n\in\mathbb N$.
\end{remark}

\medskip

As a direct consequence of Theorem~\ref{Thm1}, we present the linear version of Theorem \ref{null-3} as follows.
Consider the following linear control system:
\begin{equation}\label{appli-1}
\left\{ \begin{array}{lll}
\partial_{t}z_n-\Delta z_n+az_n=\chi_{\omega\bigcap \Omega_n}\chi_{E}u_n&\mbox{in}&\Omega_n\times (0,T),\\
z_n=0&\mbox{on}&\partial\Omega_n\times (0,T),\\
z_n(0)=z_{0}&\mbox{in}&\Omega_n,\\
\end{array}\right.\end{equation}
where $z_{0}\in L^2(\mathbb R^N)$ is an initial state and
$u_n\in L^{2}(\Omega_n\times(0,T))$  is a control.
For each $n\in\mathbb N$,  we write $z_n(\cdot; z_0, u_n)$ for the solution to (\ref{appli-1}).
By a standard duality method (see, for instance, \cite[Theorem 2.42]{Coron}),
we can easily obtain the following null controllability result with a uniform upper bound on control costs.

\begin{corollary}\label{appli-2} Under the same assumptions of  Theorem~\ref{Thm1},
for each $n\in\mathbb N$ and  $z_{0}\in L^2(\mathbb R^N)$, there is a control $u_n\in L^{2}(0,T;L^2(\mathbb{R}^N))$, with
a uniform cost
\begin{equation*}
\|u_n\|_{L^2(0,T;L^2(\mathbb{R}^N))}\leq
e^{\widetilde{C}} e^{C\left(T+T\|a\|_{\infty}+\|a\|_{\infty}^{2/3}\right)}\|z_{0}\|_{L^{2}(\mathbb R^N)},
  \end{equation*}
where the constants $C$ and $\widetilde{C}$ are given by  Theorem~\ref{Thm1}, so that $z_n(T; z_{0},u_n)=0$ in $\Omega_n$.
\end{corollary}

\subsection{Quantitative estimates of unique continuation}\label{pre}

Before giving the proof of Theorem~\ref{Thm1}, we  present the following quantitative unique continuation
property for all  solutions of \eqref{1.1}.
\begin{proposition}\label{lemma-2.2}
Let $(H_2)$ hold.
Let $0<r<R<\infty$ and  $\delta\in(0,1]$.
Then there are a universal constant $C>0$ and three positive constants $C_{1}\triangleq C_{1}(R,\delta), C_{2}\triangleq C_{2}(R,\delta)$ and $\gamma\triangleq\gamma(r,R,\delta)\in (0,1)$ so that for any $n\in\mathbb N$, any $x_{0}\in \Omega_n$, and any $\varphi_{0}\in L^2(\Omega_n)$,
the solution $\varphi_n$ of (\ref{1.1}) satisfies
\begin{eqnarray*}
\int_{B_{R}(x_{0})\bigcap \Omega_n}\!\!\!\!|\varphi_n(x,T)|^{2}\mathrm{d}x&\!\!\leq&\!\!\!\!
\left[C_{1}e^{[1+2C(1+\frac{1}{R^2})](1+\frac{4}{T}+\|a\|^{2/3}_{\infty})
+\frac{C_{2}}{T}+2T\|a\|_{\infty}}\int_{T/2}^{T}\int_{Q_{2R_{0}}(x_{0})\bigcap \Omega_n}\!\!\varphi_n^{2}(x,t)\mathrm{d}x\mathrm{d}t\right]^{\gamma}\\
\\
&&\times\left(2\int_{B_{r}(x_0)\bigcap \Omega_n}|\varphi_n(x,T)|^{2}\mathrm{d}x\right)^{1-\gamma},
\end{eqnarray*}
where $R_{0}\triangleq(1+2\delta)R$.
\end{proposition}

\begin{remark}
Note that the constants above are independent of the sizes of the domains $\Omega_n$.
\end{remark}

Here, we do not give the complete proof of Proposition \ref{lemma-2.2} since it is basically already done in \cite{duan-wang-zhang}.  We only point out the major  difference  compared with the proof of \cite[Lemma 3.2]{duan-wang-zhang}.  Indeed,  the key ingredient in the proof of \cite[Lemma 3.2]{duan-wang-zhang} is a monotonicity formula of parabolic frequency function in a bounded ball (i.e., \cite[Lemma 3.1]{duan-wang-zhang}).  Instead, the following analogous  monotonicity formula in a convex and bounded domain could be utilized  when one seeks for the detailed proof of Proposition \ref{lemma-2.2}.

\begin{proposition}\label{lemma-2.1}(\cite{Phung-Wang-1} or \cite{cz1})
Let $\Omega$ be a bounded and  convex subset in $\mathbb R^N$.
Let $r>0$, $\lambda>0$, $T>0$  and $x_{0}\in \Omega$. Denote by
$$
G_{\lambda}(x,t)\triangleq\frac{1}{(T-t+\lambda)^{N/2}}e^{-\frac{|x-x_{0}|^{2}}{4(T-t+\lambda)}},
\;\;t\in [0,T].
$$
For $u\in H^{1}(0,T; L^{2}(B_{r}(x_{0})\bigcap \Omega))\bigcap L^{2}(0,T; H^{2}(B_{r}(x_{0})\bigcap \Omega)
\bigcap H^{1}_{0}(B_{r}(x_{0})\bigcap \Omega))$ and $t\in (0,T],$ set
\begin{equation*}\label{3.111}
N_{\lambda,r}(t)\triangleq\frac{\int_{B_{r}(x_{0})\bigcap \Omega}|\nabla u(x,t)|^{2}G_{\lambda}(x,t)\mathrm{d}x}{\int_{B_{r}(x_{0})\bigcap \Omega}|u(x,t)|^{2}G_{\lambda}(x,t)\mathrm{d}x} \;\;\;\; \text{whenever}\;\; \int_{B_{r}(x_{0})\bigcap \Omega}|u(x,t)|^{2}\mathrm{d}x\neq0.
\end{equation*}
Then
\begin{equation*}
\frac{\mathrm{d}}{\mathrm{d}t}N_{\lambda,r}(t)\leq\frac{1}{T-t+\lambda}N_{\lambda,r}(t)
+\frac{\int_{B_{r}(x_{0})\bigcap \Omega}|(\partial_{t}u-\Delta u)(x,t)|^{2}G_{\lambda}(x,t)\mathrm{d}x}{\int_{B_{r}(x_{0})\bigcap \Omega}|u(x,t)|^{2}G_{\lambda}(x,t)\mathrm{d}x}.
\end{equation*}

\end{proposition}

\medskip

Based on Proposition \ref{lemma-2.2}, we then could obtain a global interpolation inequality
for solutions of \eqref{1.1} at one point of time variable.

\begin{proposition}\label{2.3} Assume that $(H_1)$ and $(H_2)$ hold. Then there are two constants $C_{3}\triangleq C_{3}(r_2)>0$ and
$\theta\triangleq\theta(r_1,r_2)\in (0,1)$ so that for any $n\in\mathbb N$ and any $\varphi_{0}\in L^{2}(\Omega_n)$,
the solution $\varphi_n$ of (\ref{1.1}) satisfies
 \begin{equation}\label{3.33333}
\int_{\Omega_n}|\varphi_n(x,T)|^{2}\mathrm{d}x\leq e^{C_{3}
\left(T^{-1}+T+T\|a\|_{\infty}+\|a\|_{\infty}^{2/3}\right)}
 \left(\int_{\Omega_n}|\varphi_{0}(x)|^{2}\mathrm{d}x\right)^{\theta}
\left(\int_{\omega\bigcap \Omega_n}|\varphi_n(x,T)|^{2}\mathrm{d}x\right)^{1-\theta}.
 \end{equation}
\end{proposition}
\begin{proof}
According to  Proposition~\ref{lemma-2.2} (where $x_0, r, R$ and $\delta$ are replaced by  $x_i\in \Omega_n, r_1, \sqrt{N}r_2$
and $1/2$, respectively), we obtain
\begin{eqnarray*}
 &&\int_{\mbox{int}(Q_{r_2}(x_{i}))}|\varphi_n(x,T)|^{2}\mathrm{d}x
 \leq \int_{B_{\sqrt{N}r_2}(x_{i})\bigcap \Omega_n}|\varphi_n(x,T)|^{2}\mathrm{d}x\\
 &\leq&\left[\mathcal{\widehat{K}}_{1}e^{[1+2C(1+r_2^{-2})](1+4T^{-1}+
 \|a\|^{2/3}_{\infty})+\mathcal{\widehat{K}}_{2}T^{-1}
 +2T\|a\|_{\infty}}\int_{T/2}^{T}\int_{Q_{4\sqrt{N}r_2}(x_{i})\bigcap \Omega_n}\varphi_n^{2}\mathrm{d}x\mathrm{d}t\right]^{\theta}\\
 &&\times\left[2\int_{B_{r_1}(x_{i})\bigcap \Omega_n}|\varphi_n(x,T)|^{2}\mathrm{d}x\right]^{1-\theta},
 \end{eqnarray*}
where $\mathcal{\widehat{K}}_{1}\triangleq\mathcal{\widehat{K}}_{1}(r_2)>0,
\mathcal{\widehat{K}}_{2}\triangleq\mathcal{\widehat{K}}_{2}(r_2)>0$ and
$\theta\triangleq\theta(r_1,r_2)\in (0,1).$ This, along with Young's  inequality,
implies that for each $\varepsilon>0,$
\begin{eqnarray*}
 &&\int_{\mbox{int}(Q_{r_2}(x_{i}))}|\varphi_n(x,T)|^{2}\mathrm{d}x\\
 &\leq&\varepsilon\theta\mathcal{\widehat{K}}_{1}e^{[1+2C(1+r_2^{-2})]
 (1+4T^{-1}+\|a\|^{2/3}_{\infty})+\mathcal{\widehat{K}}_{2}T^{-1}+2T\|a\|_{\infty}}
 \int_{T/2}^{T}\int_{Q_{4\sqrt{N}r_2}(x_{i})\bigcap \Omega_n}\varphi_n^{2}\mathrm{d}x\mathrm{d}t\\
 &&+2\varepsilon^{-\frac{\theta}{1-\theta}}(1-\theta)\int_{B_{r_1}(x_{i})\bigcap \Omega_n}
 |\varphi_n(x,T)|^{2}\mathrm{d}x.
 \end{eqnarray*}
By $(H_2)$, we have
\begin{equation}\label{3.44444}
\begin{array}{lll}
 &&\displaystyle{}\int_{\Omega_n}|\varphi_n(x,T)|^{2}\mathrm{d}x
 =\sum_{i\in I_n}\int_{\mbox{int}(Q_{r_2}(x_{i}))}|\varphi_n(x,T)|^{2}\mathrm{d}x\\
 \\
 &\leq&\displaystyle{}\varepsilon\theta\mathcal{\widehat{K}}_{1}
 e^{[1+2C(1+r_2^{-2})](1+4T^{-1}+\|a\|^{2/3}_{\infty})
 +\mathcal{\widehat{K}}_{2}T^{-1}+2T\|a\|_{\infty}}\sum_{i\in I_n}\int_{T/2}^{T}
 \int_{Q_{4\sqrt{N}r_2}(x_{i})\bigcap \Omega_n}\varphi_n^{2}\mathrm{d}x\mathrm{d}t\\
 \\
 &&\displaystyle{}+2\varepsilon^{-\frac{\theta}{1-\theta}}(1-\theta)
 \int_{\omega\bigcap \Omega_n}|\varphi_n(x,T)|^{2}\mathrm{d}x.
\end{array}
\end{equation}
Denote by
\begin{equation*}
{\widetilde{\varphi}}_n(x,t)\triangleq
\left\{
\begin{array}{lll}
\varphi_n(x,t)&\mbox{if}&(x,t)\in \Omega_n\times (0,T),\\
0&\mbox{if}&(x,t)\in (\mathbb{R}^N\setminus \Omega_n)\times (0,T).
\end{array}\right.
\end{equation*}
We can directly check that
\begin{eqnarray*}
&&\sum_{i\in I_n}\int_{T/2}^{T}
 \int_{Q_{4\sqrt{N}r_2}(x_{i})\bigcap \Omega_n}\varphi_n^{2}\mathrm{d}x\mathrm{d}t\nonumber\\
&\leq&\sum_{i\geq 1} \int_{T/2}^{T}
 \int_{Q_{4\sqrt{N}r_2}(x_{i})}(\chi_{\Omega_n}{\widetilde{\varphi}}_n)^{2}\mathrm{d}x\mathrm{d}t\\
 &\leq&\mathcal{\widehat{K}}_{3} \int_{T/2}^{T}
\int_{\mathbb{R}^{N}}(\chi_{\Omega_n}{\widetilde{\varphi}}_n)^{2}\mathrm{d}x\mathrm{d}t
=\mathcal{\widehat{K}}_{3} \int_{T/2}^{T}
\int_{\Omega_n} \varphi_n^{2}\mathrm{d}x\mathrm{d}t,\nonumber
\end{eqnarray*}
 where $\mathcal{\widehat{K}}_{3}>0$. Then it follows from (\ref{3.44444}) that
 \begin{eqnarray*}
 &&\int_{\Omega_n}|\varphi_n(x,T)|^{2}\mathrm{d}x\\
 &\leq&\varepsilon\theta\mathcal{\widehat{K}}_{1}\mathcal{\widehat{K}}_{3}
 e^{[1+2C(1+r_2^{-2})](1+4T^{-1}+\|a\|^{2/3}_{\infty})
 +\mathcal{\widehat{K}}_{2}T^{-1}+2T\|a\|_{\infty}}
 \int_{T/2}^{T}\int_{\Omega_n}\varphi_n^{2}\mathrm{d}x\mathrm{d}t\\
 &&+2\varepsilon^{-\frac{\theta}{1-\theta}}(1-\theta)
 \int_{\omega\bigcap \Omega_n}|\varphi_n(x,T)|^{2}\mathrm{d}x \;\; \mathrm{for \ each} \ \varepsilon>0.
 \end{eqnarray*}
 This implies
 \begin{equation}\label{3.444440}
\begin{array}{lll}
 &&\displaystyle{}\int_{\Omega_n}|\varphi_n(x,T)|^{2}\mathrm{d}x\\
 \\
 &\leq&\displaystyle{}\left[\mathcal{\widehat{K}}_{1}\mathcal{\widehat{K}}_{3}
 e^{[1+2C(1+r_2^{-2})](1+4T^{-1}+\|a\|^{2/3}_{\infty})
 +\mathcal{\widehat{K}}_{2}T^{-1}+2T\|a\|_{\infty}}
 \int_{T/2}^{T}\int_{\Omega_n}\varphi_n^{2}\mathrm{d}x\mathrm{d}t\right]^{\theta}\\
 \\
 &&\displaystyle{}\times\left[2\int_{\omega\bigcap \Omega_n}|\varphi_n(x,T)|^{2}\mathrm{d}x\right]^{1-\theta}.
\end{array}
\end{equation}

 Noting that
 $$
 \int_{\Omega_n} |\varphi_n(x,t)|^2\,\mathrm dx\leq e^{2\|a\|_{\infty}t}
 \int_{\Omega_n} |\varphi_0(x)|^2\,\mathrm dx\  \ \ \mathrm{for \  each} \  t\in [0,T],
 $$
 by (\ref{3.444440}), we deduce
 \begin{eqnarray*}
 \int_{\Omega_n}|\varphi_n(x,T)|^{2}\mathrm{d}x
 &\leq&\left[\mathcal{\widehat{K}}_{1}\mathcal{\widehat{K}}_{3}
 Te^{[1+2C(1+r_2^{-2})](1+4T^{-1}+\|a\|^{2/3}_{\infty})
 +\mathcal{\widehat{K}}_{2}T^{-1}+2T\|a\|_{\infty}}e^{2T\|a\|_{\infty}}
 \int_{\Omega_n}\varphi_{0}^{2}\mathrm{d}x\right]^{\theta}\\
 &&\times\left[2\int_{\omega\bigcap \Omega_n}|\varphi_n(x,T)|^{2}\mathrm{d}x\right]^{1-\theta}.
 \end{eqnarray*}
 Hence, (\ref{3.33333}) follows from the latter inequality immediately.
\end{proof}

\subsection{Proof of Theorem~\ref{Thm1}}\label{pro}

Now, we are able to present  the proof of Theorem~\ref{Thm1}  by using the telescoping series method.  The proof is similar to that of \cite[Theorem 1.1]{duan-wang-zhang}. Here we only sketch the proof.

\medskip

\noindent\textbf{Proof\ of\ Theorem~\ref{Thm1}}.
Arbitrarily fix $n\in\mathbb N$.
For any $0\leq t_{1}<t_{2}\leq T$, by a translation in the time variable and Proposition~\ref{2.3}, we obtain from Young's inequality that
 \begin{equation}\label{2019-7-9}
 \|\varphi_n(t_{2})\|^{2}_{L^{2}(\Omega_n)}\leq\varepsilon
 \|\varphi_n(t_{1})\|^{2}_{L^{2}(\Omega_n)}+
 \frac{\mathcal{\widetilde{K}}_{1}}{\varepsilon^{\alpha}}e^{\frac{\mathcal{\widetilde{K}}_{2}}{t_{2}-t_{1}}}
 \|\varphi_n(t_{2})\|^{2}_{L^{2}(\omega\bigcap \Omega_n)} \ \ \ \mathrm{for\ each}\ \varepsilon>0,
 \end{equation}
where $\mathcal{\widetilde{K}}_{1}\triangleq e^{\frac{C_3}{1-\theta}
\left(T+T\|a\|_{\infty}+\|a\|_{\infty}^{2/3}\right)},$
$\mathcal{\widetilde{K}}_{2}\triangleq C_{3}/(1-\theta)$
and $\alpha\triangleq\theta/(1-\theta)$.

Let $l$ be a density point of $E$. According to Proposition 2.1 in \cite{Phung-Wang},
for each $\kappa>1$, there exists $l_{1}\in (l,T)$, depending on $\kappa$ and $E$,
so that the sequence $\{l_{m}\}_{m\geq1}$, given by
$$
l_{m+1}=l+\frac{1}{\kappa^{m}}(l_{1}-l),
$$
satisfies
 \begin{equation}\label{3.2525251}
l_{m}-l_{m+1}\leq 3|E\bigcap(l_{m+1},l_{m})|.
 \end{equation}

Next, let $0<l_{m+2}<l_{m+1}\leq t<l_{m}<l_{1}<T$. It follows from (\ref{2019-7-9}) that
\begin{equation}\label{3.2525252}
\|\varphi_n(t)\|^{2}_{L^{2}(\Omega_n)}\leq \varepsilon\|\varphi_n(l_{m+2})\|^{2}_{L^{2}(\Omega_n)}
+\frac{\mathcal{\widetilde{K}}_{1}}{\varepsilon^{\alpha}}
e^{\frac{\mathcal{\widetilde{K}}_{2}}{t-l_{m+2}}}\|\varphi_n(t)\|^{2}_{L^{2}(\omega\bigcap \Omega_n)} \ \mathrm{for\ each}\ \varepsilon>0.
 \end{equation}
By a standard energy estimate,  we have
$$
\|\varphi_n(l_{m})\|_{L^{2}(\Omega_n)}\leq e^{T\|a\|_{\infty}}\|\varphi_n(t)\|_{L^{2}(\Omega_n)}.
$$
This, along with (\ref{3.2525252}), implies
$$
\|\varphi_n(l_{m})\|^{2}_{L^{2}(\Omega_n)}\leq e^{2T\|a\|_{\infty}}\left(\varepsilon\|\varphi_n(l_{m+2})
\|^{2}_{L^{2}(\Omega_n)}+
\frac{\mathcal{\widetilde{K}}_{1}}{\varepsilon^{\alpha}}
e^{\frac{\mathcal{\widetilde{K}}_{2}}{t-l_{m+2}}}\|\varphi_n(t)\|^{2}_{L^{2}(\omega\bigcap \Omega_n)}\right)
\ \mathrm{for\ each}\ \varepsilon>0,
$$
which indicates that
$$\|\varphi_n(l_{m})\|^{2}_{L^{2}(\Omega_n)}
\leq \varepsilon\|\varphi_n(l_{m+2})\|^{2}_{L^{2}(\Omega_n)}
+\frac{\mathcal{\widetilde{K}}_{3}}{\varepsilon^{\alpha}}
e^{\frac{\mathcal{\widetilde{K}}_{2}}{t-l_{m+2}}}\|\varphi_n(t)\|^{2}_{L^{2}(\omega\bigcap \Omega_n)}
 \ \mathrm{for\ each}\ \varepsilon>0,
 $$
where $\mathcal{\widetilde{K}}_{3}=(e^{2T\|a\|_{\infty}})^{1+\alpha}\mathcal{\widetilde{K}}_{1}$.
Integrating the latter inequality over $ E\bigcap(l_{m+1},l_{m})$ gives
\begin{equation}\label{3.2525253}
\begin{array}{lll}
 \displaystyle{}|E\bigcap(l_{m+1},l_{m})|\|\varphi_n(l_{m})\|^{2}_{L^{2}(\Omega_n)}
 &\leq&\displaystyle{}\varepsilon |E\bigcap(l_{m+1},l_{m})|\|\varphi_n(l_{m+2})\|^{2}_{L^{2}(\Omega_n)}\\
 &&\displaystyle{}+\frac{\mathcal{\widetilde{K}}_{3}}
 {\varepsilon^{\alpha}}e^{\frac{\mathcal{\widetilde{K}}_{2}}{l_{m+1}-l_{m+2}}}
 \int_{l_{m+1}}^{l_{m}}\chi_{E}\|\varphi_n(t)\|^{2}_{L^{2}(\omega\bigcap \Omega_n)}\mathrm{d}t
\end{array}
\end{equation}
for each $\varepsilon>0$.

Since $l_{m}-l_{m+1}=(\kappa-1)(l_{1}-l)/\kappa^{m},$ by (\ref{3.2525253}) and (\ref{3.2525251}), we obtain
\begin{eqnarray*}
\|\varphi_n(l_{m})\|^{2}_{L^{2}(\Omega_n)}&\leq& \frac{1}{|E\bigcap(l_{m+1},l_{m})|}
\frac{\mathcal{\widetilde{K}}_{3}}{\varepsilon^{\alpha}}
e^{\frac{\mathcal{\widetilde{K}}_{2}}{l_{m+1}-l_{m+2}}}
\int_{l_{m+1}}^{l_{m}}\chi_{E}\|\varphi_n(t)\|^{2}_{L^{2}(\omega\bigcap \Omega_n)}\mathrm{d}t+\varepsilon \|\varphi_n(l_{m+2})\|^{2}_{L^{2}(\Omega_n)}\\
&\leq&\frac{3\kappa^{m}}{(l_{1}-l)(\kappa-1)}
\frac{\mathcal{\widetilde{K}}_{3}}{\varepsilon^{\alpha}}
e^{\mathcal{\widetilde{K}}_{2}\left(\frac{1}{l_{1}-l}\frac{\kappa^{m+1}}{\kappa-1}\right)}
\int_{l_{m+1}}^{l_{m}}\chi_{E}\|\varphi_n(t)\|^{2}_{L^{2}(\omega\bigcap \Omega_n)}\mathrm{d}t+
\varepsilon \|\varphi_n(l_{m+2})\|^{2}_{L^{2}(\Omega_n)}
 \end{eqnarray*}
 for each $\varepsilon>0$.
This yields
\begin{equation}
\begin{array}{lll}
&&\|\varphi_n(l_{m})\|^{2}_{L^{2}(\Omega_n)}\\
&\leq& \displaystyle{}
 \frac{1}{\varepsilon^{\alpha}}\frac{3}{\kappa}
 \frac{\mathcal{\widetilde{K}}_{3}}{\mathcal{\widetilde{K}}_{2}}
 e^{2\mathcal{\widetilde{K}}_{2}\left(\frac{1}{l_{1}-l}
 \frac{\kappa^{m+1}}{\kappa-1}\right)}\int_{l_{m+1}}^{l_{m}}\chi_{E}
 \|\varphi_n(t)\|^{2}_{L^{2}(\omega\bigcap \Omega_n)}\mathrm{d}t\displaystyle{}+
 \varepsilon \|\varphi_n(l_{m+2})\|^{2}_{L^{2}(\Omega_n)}
\end{array}\label{3.2525254}
\end{equation}
for each $\varepsilon>0$.
Denote by $d\triangleq 2\mathcal{\widetilde{K}}_{2}/[\kappa(l_{1}-l)(\kappa-1)]$.
It follows from (\ref{3.2525254}) that
\begin{eqnarray*}
\varepsilon^{\alpha}e^{-d\kappa^{m+2}}\|\varphi_n(l_{m})\|^{2}_{L^{2}(\Omega_n)}
-\varepsilon^{1+\alpha}e^{-d\kappa^{m+2}}\|\varphi_n(l_{m+2})\|^{2}_{L^{2}(\Omega_n)}
\leq\frac{3}{\kappa}\frac{\mathcal{\widetilde{K}}_{3}}
{\mathcal{\widetilde{K}}_{2}}\int_{l_{m+1}}^{l_{m}}\chi_{E}\|\varphi_n(t)\|^{2}_{L^{2}(\omega\bigcap \Omega_n)}\mathrm{d}t
\end{eqnarray*}
for each $\varepsilon>0$.

Choosing $\varepsilon=e^{-d\kappa^{m+2}}$ in the above inequality gives
\begin{equation}\label{3.25252555}
\begin{array}{lll}
 &&\displaystyle{}e^{-(1+\alpha)d\kappa^{m+2}}\|\varphi_n(l_{m})\|^{2}_{L^{2}(\Omega_n)}
 -e^{-(2+\alpha)d\kappa^{m+2}}\|\varphi_n(l_{m+2})\|^{2}_{L^{2}(\Omega_n)}\\
 &\leq&\displaystyle{}\frac{3}{\kappa}\frac{\mathcal{\widetilde{K}}_{3}}
 {\mathcal{\widetilde{K}}_{2}}\int_{l_{m+1}}^{l_{m}}\chi_{E}\|\varphi_n(t)\|^{2}_{L^{2}(\omega\bigcap \Omega_n)}\mathrm{d}t.
\end{array}
\end{equation}
Taking $\kappa=\sqrt{(\alpha+2)/(\alpha+1)}$ in (\ref{3.25252555}), we then have
\begin{eqnarray*}
e^{-(2+\alpha)d\kappa^{m}}\|\varphi_n(l_{m})\|^{2}_{L^{2}(\Omega_n)}
-e^{-(2+\alpha)d\kappa^{m+2}}\|\varphi_n(l_{m+2})\|^{2}_{L^{2}(\Omega_n)}
\leq \frac{3}{\kappa}\frac{\mathcal{\widetilde{K}}_{3}}{\mathcal{\widetilde{K}}_{2}}
\int_{l_{m+1}}^{l_{m}}\chi_{E}\|\varphi_n(t)\|^{2}_{L^{2}(\omega\bigcap \Omega_n)}\mathrm{d}t.
\end{eqnarray*}
Changing $m$ to $2m'$ and summing the above inequality from $m'=1$ to $\infty$ give the desired result. This finishes the proof of Theorem~\ref{Thm1}.
\qed

\vskip 5mm

\section{Proof of Theorem \ref{null-3}}\label{b2}

By a density argument, we can assume that $f\in C^1$. We will use  the linearized result (i.e., Corollary \ref{appli-2}) and the
 Kakutani-Fan-Glicksberg  fixed point theorem
 (see, e.g., \cite[Theorem 1.14]{Wang-Wang-Xu-Zhang})  to prove Theorem \ref{null-3}.

To this end,  we first define
\begin{equation*}
a(r)\triangleq\left\{
\begin{array}{lll}
\displaystyle{\frac{f(r)}{r}}&\mbox{if}&r\not=0,\\
f^\prime(0)&\mbox{if}&r=0.
\end{array}\right.
\end{equation*}
By $(H_3)$, we have that
\begin{equation}\label{null-4}
|a(r)|\leq L\;\;\mbox{for all}\;\;r\in \mathbb{R}.
\end{equation}

For each $n\geq1$, we set
\begin{equation*}
\mathcal{K}_n\triangleq\{\xi\in L^2(0,T;L^2(\Omega_n))\bigm|\|\xi\|_{L^2(0,T;H_0^1(\Omega_n))}
+\|\xi\|_{W^{1,2}(0,T;H^{-1}(\Omega_n))}\leq c_0\},
\end{equation*}
where $c_0>0$ will be determined later.  For each $\xi\in \mathcal{K}_n$,
we consider the following linear equation:
\begin{equation}\label{null-5}
\left\{
\begin{array}{lll}
\partial_t z-\Delta z+a(\xi(x,t))z=\chi_{\omega\bigcap \Omega_n}\chi_E u&\mbox{in}&\Omega_n\times (0,T),\\
z=0&\mbox{on}&\partial\Omega_n\times (0,T),\\
z(0)=z_0&\mbox{in}&\Omega_n.
\end{array}\right.
\end{equation}
We simply write $z(\cdot)$ for the solution of (\ref{null-5}). According to Corollary~\ref{appli-2},
there is a positive constant $\kappa\triangleq \kappa(r_1,r_2,E,T,L)$ (independent of $z_0$, $n$ and $\xi$) and a control $u$ so that
\begin{equation}\label{null-6}
\|u\|_{L^2(0,T;L^2(\mathbb{R}^N))}\leq \kappa\|z_0\|_{L^2(\mathbb R^N)}\;\;\mbox{and}\;\;z(T)=0.
\end{equation}

For each $n\geq1$, we next define a  set-valued mapping
$$
\Phi_n: \mathcal{K}_n\rightarrow 2^{L^2(0,T;L^2(\Omega_n))}
$$
by setting
\begin{eqnarray*}
\Phi_n(\xi)\triangleq\Big\{z\bigm|\mbox{there exists a control}\;\;u\in L^2(0,T;L^2(\mathbb{R}^N))\\
\mbox{so that (\ref{null-5})
and (\ref{null-6}) hold}\Big\},\;\; \xi\in \mathcal{K}_n.
\end{eqnarray*}
One can easily check that  $\Phi_n(\xi)\not=\emptyset$ for each $\xi\in \mathcal{K}_n$.

\medskip

The rest of the proof will be organized by several steps as follows.

\medskip

Step 1. We show that $\mathcal{K}_n$ is
compact and convex in $L^2(0,T;L^2(\Omega_n))$, and that each $\Phi_n(\xi)$ is convex in $L^2(0,T;L^2(\Omega_n))$.

These can be directly checked.\\

Step 2. We claim that $\Phi_n(\mathcal{K}_n)\subset
\mathcal{K}_n$.

Given $\xi\in
\mathcal{K}_n$, there is a control  $u$ satisfying (\ref{null-5}) and (\ref{null-6}).
By a standard energy estimate method
and by (\ref{null-4})-(\ref{null-6}), we can easily check that
\begin{equation*}
\|z\|_{L^2(0,T;H_0^1(\Omega_n))}+\|z\|_{W^{1,2}(0,T;H^{-1}(\Omega_n))}\leq
c_1\|z_0\|_{L^2(\Omega_n)},
\end{equation*}
for a positive constant $c_1$ (independent of $z_0$, $n$ and $\xi$). Hence,
$$
z\in \mathcal{K}_n\;\;\mbox{if}\;\;c_0=c_1\|z_0\|_{L^2(\Omega_n)}.
$$

Step 3. We show that  Graph($\Phi_n$) is closed.

It suffices to show that $z\in \Phi_n(\xi)$, provided that
$$
\xi_\ell\in \mathcal{K}_n\rightarrow \xi\;\;\mbox{ strongly
in}\;\;L^2(0,T;L^2(\Omega_n))
$$
and
$$
z_\ell\in
\Phi_n(\xi_\ell)\rightarrow z\;\;\mbox{ strongly
in}\;\;L^2(0,T;L^2(\Omega_n)).
$$
 To this end,  we first observe that $\xi\in \mathcal{K}_n$, since
$\mathcal{K}_n$ is convex and closed. Next we
 claim that there exists a subsequence of
$\{\ell\}_{\ell\geq 1}$,  denoted in the same manner, so that
\begin{equation}\label{null-7}
a(\xi_\ell)z_\ell\rightarrow a(\xi)z\;\;\mbox{strongly in}\;\;L^2(0,T;L^2(\Omega_n)).
\end{equation}
Indeed, since
\begin{equation*}
\xi_\ell\rightarrow \xi\;\;\mbox{strongly in}\;\;L^2(0,T;L^2(\Omega_n)),
\end{equation*}
 we have  a subsequence of $\{\ell\}_{\ell\geq 1}$, still
 denoted by itself, so that
\begin{equation*}
\xi_\ell(x,t)\rightarrow \xi(x,t)\;\;\mbox{for a.e.}\;\;(x,t)\in \Omega_n\times (0,T).
\end{equation*}
Then, by the definition of the function $a$,
we conclude that
\begin{equation*}
a(\xi_\ell(x,t))\rightarrow a(\xi(x,t))\;\;\mbox{for a.e.}\;\;(x,t)\in \Omega_n\times (0,T).
\end{equation*}
By this and (\ref{null-4}), we can apply the Lebesgue
Dominated Convergence Theorem
to obtain
 that
\begin{eqnarray*}
&&\|a(\xi_\ell)z_\ell-a(\xi)z\|_{L^2(0,T;L^2(\Omega_n))}^2\\
&\leq&2\|a(\xi_\ell)(z_\ell-z)\|_{L^2(0,T;L^2(\Omega_n))}^2+2\|(a(\xi_\ell)-a(\xi))z\|_{L^2(0,T;L^2(\Omega_n))}^2\\
&\leq&2 L^2\|z_\ell-z\|_{L^2(0,T;L^2(\Omega_n))}^2+2\|(a(\xi_\ell)-a(\xi))z\|_{L^2(0,T;L^2(\Omega_n))}^2\\
&\rightarrow&0.
\end{eqnarray*}
This leads to  (\ref{null-7}).

Finally,  for each $\ell\geq 1$, since $z_\ell\in \Phi_n(\xi_\ell)\subset \mathcal{K}_n$, there is $u_\ell\in
L^2(0,T;L^2(\mathbb{R}^N))$ with
\begin{equation}\label{null-8}
\|u_\ell\|_{L^2(0,T;L^2(\mathbb{R}^N))}\leq
\kappa\|z_0\|_{L^2(\mathbb R^N)},
\end{equation}
so that
\begin{equation}\label{null-9}
\left\{
\begin{array}{lll}
\partial_t z_\ell-\Delta z_\ell+a(\xi_\ell(x,t))z_\ell
=\chi_{\omega\bigcap \Omega_n}\chi_E u_\ell&\mbox{in}&\Omega_n\times (0,T),\\
z_\ell=0&\mbox{on}&\partial\Omega_n\times (0,T),\\
z_\ell(0)=z_0&\mbox{in}&\Omega_n,\\
z_\ell(T)=0&\mbox{in}&\Omega_n
\end{array}\right.
\end{equation}
and
\begin{equation}\label{null-10}
\|z_\ell\|_{L^2(0,T;H_0^1(\Omega_n))}+
\|z_\ell\|_{W^{1,2}(0,T;H^{-1}(\Omega_n))}\leq c_0.
\end{equation}
According to (\ref{null-8}) and (\ref{null-10}), there is a
control $u$ and a subsequence of $\{\ell\}_{\ell\geq 1}$,  denoted in the same manner, so that
\begin{equation}\label{null-11}
u_\ell\rightarrow u\;\;\mbox{weakly  in}\;\;L^2(0,T;L^2(\mathbb{R}^N)),
\end{equation}
\begin{equation}\label{null-12}
z_\ell\rightarrow z\;\;\mbox{weakly in}\;\;L^2(0,T;H_0^1(\Omega_n))\bigcap W^{1,2}(0,T;H^{-1}(\Omega_n)),\\
\end{equation}
and
\begin{equation}\label{null-13}
z_\ell(T)\rightarrow z(T)\;\;\mbox{strongly in}\;\;L^2(\Omega_n).
\end{equation}
Passing to the limit for $\ell\rightarrow +\infty$ in
(\ref{null-8}) and (\ref{null-9}), making use of
(\ref{null-7}) and (\ref{null-11})-(\ref{null-13}),
we obtain that $z\in \Phi_n(\xi)$.\\

Step 4. We apply the
 Kakutani-Fan-Glicksberg Theorem to end the proof.

From the conclusions in the above three steps, we find that the map $\Phi_n$ satisfies
conditions of  the
 Kakutani-Fan-Glicksberg Theorem. Thus we can apply this theorem
to conclude that there exists  $z\in\mathcal K_n$ so that $z\in \Phi_n(z)$.
Then, by the definition of $\Phi_n$ and using the fact that
\begin{equation*}
a(z(x,t))z(x,t)=f(z(x,t)),
\end{equation*}
one can finish the proof of Theorem~\ref{null-3}.
\qed

\section{Proof of Theorem \ref{null-2}}\label{b5}
For the sake of completeness\footnote{Although we believe that the existence and uniqueness of weak solutions for nonlinear heat equations in $\mathbb R^N$ have been well established in the literature, we did not yet find the exact reference with precise proofs.}, we first consider the well-posedness of the following non-homogeneous  semilinear heat equation  in $\mathbb R^N$:
\begin{equation}\label{null-1-16}
\left\{
\begin{array}{lll}
\partial_t y-\Delta y+f(y)=g&\mbox{in}&\mathbb{R}^N\times (0,T),\\
y(0)=y_0&\mbox{in}&\mathbb{R}^N,
\end{array}\right.
\end{equation}
where $y_0\in L^2(\mathbb{R}^N)$, $g\in L^2(\mathbb{R}^N\times (0,T))$,
and  the nonlinearity $f$ satisfies the assumption $(H_3)$.

\begin{Definition}\label{weak solution}
We say that $y$ is a weak solution of (\ref{null-1-16}) if
 \begin{description}
\item[($i$)] $y\in L^2(0,T;H^1(\mathbb{R}^N))\bigcap W^{1,2}(0,T;H^{-1}(\mathbb{R}^N))\subset C([0,T];L^2(\mathbb{R}^N))$
and $y(0)=y_0$;

\item[($ii$)] For each $\psi\in H^1(\mathbb{R}^N)$, the following equality holds:
\begin{equation*}
\langle \partial_t y(t), \psi\rangle_{H^{-1}(\mathbb{R}^N),H^1(\mathbb{R}^N)}
+\langle \nabla y(t), \nabla \psi\rangle_{L^2(\mathbb{R}^N)}
+\langle f(y(t)), \psi\rangle_{L^2(\mathbb{R}^N)}
=\langle g(t), \psi\rangle_{L^2(\mathbb{R}^N)}
\end{equation*}
for a.e. $t\in (0,T)$.
\end{description}
\end{Definition}

Before presenting the proof of Theorem~\ref{null-2},  we first state the well-posedness of (\ref{null-1-16}) in the sense of the above definition.

\begin{Lemma}\label{Existence}
Under the assumption $(H_3)$,  for each $y_0\in L^2(\mathbb{R}^N)$ and each $g\in L^2(0,T;L^2(\mathbb{R}^N))$, the equation (\ref{null-1-16}) has a unique weak solution.
\end{Lemma}

\begin{remark}
The argument below is also effective in our proof of Theorem \ref{null-2}. Indeed, it is inspired by the approach in \cite[Section 4.4]{RRS16} constructing weak solutions for the Navier-Stokes equations on the whole space.  The main idea there is to use solutions of the Navier-Stokes equations on a sequence of expanding bounded domains as a sequence of approximate solutions on the whole space and show the convergence of such
solutions.
\end{remark}

\begin{proof} For each $n\geq 1$, we shall consider the equation
\begin{equation}\label{null-14}
\left\{
\begin{array}{lll}
\partial_t y_n-\Delta y_n+f(y_n)=g&\mbox{in}&\Omega_n\times (0,T),\\
y_n=0&\mbox{on}&\partial\Omega_n\times (0,T),\\
y_n(0)=y_0&\mbox{in}&\Omega_n,
\end{array}\right.
\end{equation}
where $\Omega_n$ is constructed  as in $(H_2)$. By $(H_3)$, (\ref{null-14}), and
 a standard energy estimate method, we can easily check that
\begin{equation}\label{null-17}
\|y_n\|_{L^2(0,T;H_0^1(\Omega_n))}+\|y_n\|_{W^{1,2}(0,T;H^{-1}(\Omega_n))}\leq C.
\end{equation}
Here and throughout the proof of this lemma, $C$ denotes a  positive constant independent of $n$.

We extend $y_n$ to $\mathbb{R}^N\times (0,T)$ by zero and still denote this extension by $y_n$.
On one hand, by (\ref{null-17}), there is a subsequence of $\{n\}_{n\geq 1}$, still denoted by itself,
and $y^*\in L^2(0,T;H^1(\mathbb{R}^N))$, so that
\begin{equation}\label{null-18}
y_n\rightarrow y^*\;\;\mbox{weakly in}\;\;L^2(0,T;H^1(\mathbb{R}^N)).
\end{equation}
On the other hand, for each $M>0$, there exists a positive integer $n_0(M)$ so that
\begin{equation*}
H^{-1}(\Omega_n)\subset H^{-1}(B_M(0))\;\;
\mbox{and}\;\;\|\cdot\|_{H^{-1}(B_M(0))}\leq \|\cdot\|_{H^{-1}(\Omega_n)}
\;\;\mbox{for all}\;\;n\geq n_0(M).
\end{equation*}
These, along with (\ref{null-17}) and (\ref{null-18}), imply that
$y^*\in W^{1,2}(0,T;H^{-1}(B_M(0)))$ and there is a subsequence of $\{n\}_{n\geq 1}$,
denoted in the same manner, so that
\begin{equation}\label{null-20}
y_n\rightarrow y^*\;\;\mbox{weakly in}\;\;L^2(0,T;H^1(\mathbb{R}^N))\bigcap W^{1,2}(0,T;H^{-1}(B_M(0))),
\end{equation}
and
\begin{equation}\label{null-20-1}
y_n\rightarrow y^*\;\;\mbox{strongly in}\;\;L^2(0,T;L^2(B_M(0))).
\end{equation}
It follows from $(H_3)$ and (\ref{null-20-1}) that
\begin{equation}\label{null-21}
f(y_n)\rightarrow f(y^*)\;\;\mbox{strongly in}\;\;L^2(0,T;L^2(B_M(0))).
\end{equation}

Arbitrarily fix $\varphi\in C_0^\infty(\mathbb{R}^N\times(0,T))$.
Let $M>0$ be large enough so that the support of $\varphi$ is contained in
$B_M(0)\times(0,T)$. Then for all $n\geq n_0(M)$, we have that
\begin{equation}\label{null-22}
\begin{array}{lll}
&&-\displaystyle{\int_0^T} \langle y_n, \partial_t \varphi\rangle_{L^2(B_M(0))}\mathrm dt
+\displaystyle{\int_0^T} \langle \nabla y_n, \nabla \varphi\rangle_{L^2(B_M(0))}\mathrm dt
+\displaystyle{\int_0^T} \langle f(y_n), \varphi\rangle_{L^2(B_M(0))}\mathrm dt\\
&=&\displaystyle{\int_0^T} \langle g, \varphi\rangle_{L^2(B_M(0))}\mathrm dt.
\end{array}
\end{equation}
Passing to the limit for $n\rightarrow \infty$ in (\ref{null-22}), by (\ref{null-20}) and (\ref{null-21}),
we obtain that
\begin{equation*}
\begin{array}{lll}
&&-\displaystyle{\int_0^T} \int_{\mathbb R^N} y^* \partial_t \varphi\;\mathrm dx\mathrm dt
+\displaystyle{\int_0^T} \int_{\mathbb R^N} \nabla y^*\cdot\nabla \varphi\;\mathrm dx\mathrm dt
+\displaystyle{\int_0^T} \int_{\mathbb R^N} f(y^*) \varphi\;\mathrm dx\mathrm dt\\
&=&\displaystyle{\int_0^T} \int_{\mathbb R^N}g \varphi\;\mathrm dx\mathrm dt,
\end{array}
\end{equation*}
which indicates that
\begin{equation}\label{null-23}
\partial_t y^*=\Delta y^*-f(y^*)+g\;\;\mbox{in the sense of distribution}.
\end{equation}
Since $\Delta y^*-f(y^*)+g\in L^2(0,T;H^{-1}(\mathbb{R}^N))$, it follows from (\ref{null-23})
that $\partial_t y^*\in  L^2(0,T;H^{-1}(\mathbb{R}^N))$. Hence,
\begin{equation}\label{null-23-1}
y^*\in  L^2(0,T;H^1(\mathbb{R}^N))\bigcap W^{1,2}(0,T;H^{-1}(\mathbb{R}^N))\subset C([0,T];L^2(\mathbb{R}^N)).
\end{equation}
This, along with (\ref{null-23}) and a density argument, implies that for each $\psi\in H^1(\mathbb{R}^N)$,
\begin{equation}\label{null-23-2}
\langle \partial_t y^*(t), \psi\rangle_{H^{-1}(\mathbb{R}^N),H^1(\mathbb{R}^N)}
+\langle \nabla y^*(t), \nabla \psi\rangle_{L^2(\mathbb{R}^N)}
+\langle f(y^*(t)), \psi\rangle_{L^2(\mathbb{R}^N)}
=\langle g(t), \psi\rangle_{L^2(\mathbb{R}^N)}
\end{equation}
for a.e. $t\in (0,T)$.\\

We next show that
\begin{equation}\label{null-23-3}
y^*(0)=y_0.
\end{equation}
To this end, we arbitrarily fix $\psi\in C_0^\infty(\mathbb{R}^N)$.
Let $M>0$ be large enough so that the support of $\psi$ is contained in
$B_M(0)$. Set $\varphi(x,t)\triangleq (T-t)\psi(x)/T$. Then for all $n\geq n_0(M)$, multiplying
both sides of (\ref{null-14}) by $\varphi$ and integrating it over $\mathbb{R}^N\times (0,T)$, we obtain that
\begin{equation}\label{null-24}
\begin{array}{lll}
&&-\displaystyle{\int_0^T} \langle y_n, \partial_t \varphi\rangle_{L^2(B_M(0))}\mathrm dt
+\displaystyle{\int_0^T} \langle \nabla y_n, \nabla \varphi\rangle_{L^2(B_M(0))}\mathrm dt
+\displaystyle{\int_0^T} \langle f(y_n), \varphi\rangle_{L^2(B_M(0))}\mathrm dt\\
&=&\displaystyle{\int_0^T} \langle g, \varphi\rangle_{L^2(B_M(0))}\mathrm dt
+\langle y_0, \psi\rangle_{L^2(B_M(0))}.
\end{array}
\end{equation}
Passing to the limit for $n\rightarrow \infty$ in (\ref{null-24}), by (\ref{null-20}) and (\ref{null-21}),
we obtain that
\begin{equation}\label{null-25}
\begin{array}{lll}
&&-\displaystyle{\int_0^T} \langle y^*, \partial_t \varphi\rangle_{L^2(\mathbb{R}^N)}\mathrm dt
+\displaystyle{\int_0^T} \langle \nabla y^*, \nabla \varphi\rangle_{L^2(\mathbb{R}^N)}\mathrm dt
+\displaystyle{\int_0^T} \langle f(y^*), \varphi\rangle_{L^2(\mathbb{R}^N)}\mathrm dt\\
&=&\displaystyle{\int_0^T} \langle g, \varphi\rangle_{L^2(\mathbb{R}^N)}\mathrm dt
+\langle y_0, \psi\rangle_{L^2(\mathbb{R}^N)}.
\end{array}
\end{equation}
Moreover, multiplying both sides of (\ref{null-23-2}) by $(T-t)/T$ and integrating it over $(0,T)$, we have that
\begin{equation}\label{null-26}
\begin{array}{lll}
&&-\displaystyle{\int_0^T} \langle y^*, \partial_t \varphi\rangle_{L^2(\mathbb{R}^N)}\mathrm dt
+\displaystyle{\int_0^T} \langle \nabla y^*, \nabla \varphi\rangle_{L^2(\mathbb{R}^N)}\mathrm dt
+\displaystyle{\int_0^T} \langle f(y^*), \varphi\rangle_{L^2(\mathbb{R}^N)}\mathrm dt\\
&=&\displaystyle{\int_0^T} \langle g, \varphi\rangle_{L^2(\mathbb{R}^N)}\mathrm dt
+\langle y^*(0), \psi\rangle_{L^2(\mathbb{R}^N)}.
\end{array}
\end{equation}
It follows from (\ref{null-25}) and (\ref{null-26}) that
\begin{equation*}
\langle y^*(0)-y_0,\psi\rangle_{L^2(\mathbb{R}^N)}=0\;\;\mbox{for each}\;\;\psi\in C_0^\infty(\mathbb{R}^N),
\end{equation*}
which indicates (\ref{null-23-3}).\\

By (\ref{null-23-1})--(\ref{null-23-3}) and Definition~\ref{weak solution},
we see that $y^*$ is a weak solution of (\ref{null-1-16}). Finally, we show the uniqueness of the weak solution.
WLOG, we assume that $\widetilde{y}$ is also a weak solution of (\ref{null-1-16}). According to Definition~\ref{weak solution},
it holds that
\begin{equation}\label{null-wang-20}
\begin{array}{l}
\langle \partial_t (y^*-\widetilde{y})(t), (y^*-\widetilde{y})(t)\rangle_{H^{-1}(\mathbb{R}^N),H^1(\mathbb{R}^N)}\\
+\langle \nabla (y^*-\widetilde{y})(t), \nabla (y^*-\widetilde{y})(t)\rangle_{L^2(\mathbb{R}^N)}
+\langle f(y^*(t))-f(\widetilde{y}(t)), (y^*-\widetilde{y})(t)\rangle_{L^2(\mathbb{R}^N)}=0
\end{array}
\end{equation}
for a.e. $t\in (0,T)$, and $y^*(0)=\widetilde{y}(0)=y_0$. Integrating (\ref{null-wang-20}) over $(0,t)$, we obtain that
\begin{equation*}
\|(y^*-\widetilde{y})(t)\|_{L^2(\mathbb{R}^N)}^2\leq 2 L \int_0^t \|(y^*-\widetilde{y})(s)\|_{L^2(\mathbb{R}^N)}^2\mathrm ds.
\end{equation*}
By Gronwall's inequality, we obtain from the latter inequality that
$y^*=\widetilde{y}$.\\

In summary, we finish the proof of Lemma~\ref{Existence}.
\end{proof}

Now, we are able to present the proof of Theorem~\ref{null-2}.\\

\noindent\textbf{Proof\ of\ Theorem~\ref{null-2}}. For each $n\geq 1$, according to Theorem~\ref{null-3},
there is a control $u_n\in L^2(0,T;L^2(\mathbb{R}^N))$ so that
\begin{equation}\label{null-27}
\left\{
\begin{array}{lll}
\partial_t y_n-\Delta y_n+f(y_n)=\chi_\omega  \chi_E u_n&\mbox{in}&\Omega_n\times (0,T),\\
y_n=0&\mbox{on}&\partial\Omega_n\times (0,T),\\
y_n(0)=y_0&\mbox{in}&\Omega_n,
\end{array}\right.
\end{equation}
\begin{equation}\label{null-28}
\|u_n\|_{L^2(0,T;L^2(\mathbb{R}^N))}\leq \kappa \|y_0\|_{L^2(\mathbb{R}^N)}
\end{equation}
and
\begin{equation}\label{null-29}
y_n(T)=0\;\;\mbox{in}\;\;\Omega_n.
\end{equation}
By a standard energy estimate method, $(H_3)$, (\ref{null-27}) and (\ref{null-28}),
we can easily check that
\begin{equation}\label{null-30}
\|y_n\|_{L^2(0,T;H_0^1(\Omega_n))}+\|y_n\|_{W^{1,2}(0,T;H^{-1}(\Omega_n))}\leq C,
\end{equation}
where $C$ is a positive constant independent of $n$.

We extend $y_n$ to $\mathbb{R}^N\times (0,T)$ by $0$ and still denote this extension by $y_n$.
By (\ref{null-30}) and (\ref{null-28}), there is a subsequence of $\{n\}_{n\geq 1}$, still denoted by itself,
and $(u^*, y^*)\in L^2(0,T;L^2(\mathbb{R}^N))\times L^2(0,T;H^1(\mathbb{R}^N))$, so that
\begin{equation}\label{null-31}
y_n\rightarrow y^*\;\;\mbox{weakly in}\;\;L^2(0,T;H^1(\mathbb{R}^N))
\end{equation}
and
\begin{equation}\label{null-32}
u_n\rightarrow u^*\;\;\mbox{weakly in}\;\;L^2(0,T;L^2(\mathbb{R}^N)).
\end{equation}
On one hand, by similar arguments as those in Lemma~\ref{Existence}, we observe that $y^*$ is the unique
weak solution of (\ref{null-wang}) (where $u$ is replaced by $u^*$), and for each $M>0$, there is a subsequence of
$\{n\}_{n\geq 1}$, denoted in the same manner, so that
\begin{equation}\label{null-33}
f(y_n)\rightarrow f(y^*)\;\;\mbox{strongly in}\;\; L^2(0,T;L^2(B_M(0))).
\end{equation}
On the other hand, we arbitrarily fix $\psi\in C_0^\infty(\mathbb{R}^N)$.
Let $M>0$ be large enough so that the support of $\psi$ is contained in
$B_M(0)$. Set $\varphi(x,t)\triangleq t\psi(x)/T$. Then for all $n\geq n_0(M)$ (where $n_0(M)$ is
the same integer as that in Lemma~\ref{Existence}), multiplying
both sides of (\ref{null-27}) by $\varphi$ and integrating it over $\mathbb{R}^N\times (0,T)$, we obtain that
\begin{equation}\label{null-34}
\begin{array}{lll}
&&-\displaystyle{\int_0^T} \langle y_n, \partial_t \varphi\rangle_{L^2(B_M(0))}\mathrm dt
+\displaystyle{\int_0^T} \langle \nabla y_n, \nabla \varphi\rangle_{L^2(B_M(0))}\mathrm dt
+\displaystyle{\int_0^T} \langle f(y_n), \varphi\rangle_{L^2(B_M(0))}\mathrm dt\\
&=&\displaystyle{\int_0^T} \langle \chi_\omega  \chi_E u_n, \varphi\rangle_{L^2(B_M(0))}\mathrm dt
-\langle y_n(T), \psi\rangle_{L^2(B_M(0))}.
\end{array}
\end{equation}
Passing to the limit for $n\rightarrow \infty$ in (\ref{null-34}), by (\ref{null-29}) and (\ref{null-31})-(\ref{null-33}),
we obtain that
\begin{equation}\label{null-35}
\begin{array}{lll}
&&-\displaystyle{\int_0^T} \langle y^*, \partial_t \varphi\rangle_{L^2(\mathbb{R}^N)}\mathrm dt
+\displaystyle{\int_0^T} \langle \nabla y^*, \nabla \varphi\rangle_{L^2(\mathbb{R}^N)}\mathrm dt
+\displaystyle{\int_0^T} \langle f(y^*), \varphi\rangle_{L^2(\mathbb{R}^N)}\mathrm dt\\
&=&\displaystyle{\int_0^T} \langle \chi_\omega  \chi_E u^*, \varphi\rangle_{L^2(\mathbb{R}^N)}\mathrm dt.
\end{array}
\end{equation}
Moreover, since $y^*$ is the weak solution of (\ref{null-wang}) (where $u$ is replaced by $u^*$), we have that
\begin{equation}\label{null-36}
\begin{array}{lll}
&&-\displaystyle{\int_0^T} \langle y^*, \partial_t \varphi\rangle_{L^2(\mathbb{R}^N)}\mathrm dt
+\displaystyle{\int_0^T} \langle \nabla y^*, \nabla \varphi\rangle_{L^2(\mathbb{R}^N)}\mathrm dt
+\displaystyle{\int_0^T} \langle f(y^*), \varphi\rangle_{L^2(\mathbb{R}^N)}\mathrm dt\\
&=&\displaystyle{\int_0^T} \langle \chi_\omega  \chi_E u^*, \varphi\rangle_{L^2(\mathbb{R}^N)}\mathrm dt
-\langle y^*(T), \psi\rangle_{L^2(\mathbb{R}^N)}.
\end{array}
\end{equation}
It follows from (\ref{null-35}) and (\ref{null-36}) that
$\langle y^*(T),\psi\rangle_{L^2(\mathbb{R}^N)}=0$ for each $\psi\in C_0^\infty(\mathbb{R}^N)$.
This implies that $y^*(T)=0$.

Hence, we finish the proof of Theorem~\ref{null-2}.
\qed

\bigskip

\noindent\textbf{Acknowledgments}.
This work was partially supported by the National Natural Science Foundation of China under grants 11771344 and 11971363.  The second author is also partially supported by the Academic Team Building Plan for Young Scholars from Wuhan University under grant 413100085.

\end{document}